\documentclass[12pt, a4paper]{article}
\usepackage[english]{babel}
\usepackage{epsfig}
\usepackage{amsmath}
\usepackage{amsthm}
\usepackage{amssymb}
\usepackage{graphicx}

\newtheorem{theorem}{Theorem}[section]
\newtheorem{definition}[theorem]{Definition}
\newtheorem{lemma}[theorem]{Lemma}
\newtheorem{proposition}[theorem]{Proposition}
\newtheorem{corollary}[theorem]{Corollary}
\newtheorem{remark}[theorem]{Remark}
\newtheorem{example}[theorem]{Example}

\newcommand{\hh}{{\mathbb{H}}}
\newcommand{\s}{{\mathbb{S}}}
\newcommand{\cc}{{\mathbb{C}}}
\newcommand{\rr}{{\mathbb{R}}}
\newcommand{\nn}{{\mathbb{N}}}

\title{\bf A new series expansion for slice regular functions}

\author{Caterina Stoppato\footnote{Partially supported by GNSAGA of the INdAM and by FIRB ``Geometria Differenziale Complessa e Dinamica Olomorfa''.} \\ 
\normalsize Universit\`a degli Studi di Milano\\
\normalsize Dipartimento di Matematica ``F. Enriques''\\
\normalsize Via Saldini 50, 20133 Milano, Italy\\  
\normalsize caterina.stoppato@unimi.it}

\date{  }
\begin{document}


\maketitle

\begin{abstract}
\noindent A promising theory of quaternion-valued functions of one quaternionic variable, now called \emph{slice regular} functions, has been introduced in \cite{cras,advances}. The basic examples of slice regular functions are the power series of type $\sum_{n \in \mathbb{N}} q^n a_n$ on their balls of convergence $B(0,R) = \{q \in \mathbb{H} : |q|<R\}$. Conversely, if $f$ is a slice regular function on a domain $\Omega \subseteq \mathbb{H}$ then it admits at each point $q_0 \in \Omega$ an expansion of type $f(q) = \sum_{n \in \mathbb{N}} (q-q_0)^{*n} a_n$ where $(q-q_0)^{*n}$ denotes the $n$th power of $q-q_0$ with respect to an appropriately defined multiplication $*$. However, the information provided by such an expansion is somewhat limited by a fact: if $q_0$ does not lie on the real axis then the set of convergence of the series in the previous equation needs not be a Euclidean neighborhood of $q_0$. We are now able to construct a new type of expansion that is not affected by this phenomenon: an expansion into series of polynomials valid in open subsets of the domain. Along with this construction, we present applications to the computation of the multiplicities of zeros and of partial derivatives.
\end{abstract}

\vfill

\section{Introduction}\label{sectionintroduction}

Let $\hh$ denote the real algebra of quaternions, that is the vector space $\rr^4$ endowed with the multiplication constructed as follows: if $1,i,j,k$ denotes the standard basis, define $$i^2 = j^2 = k^2 = -1,$$  $$ ij = -ji = k, jk = -kj = i, ki = -ik = j,$$ let $1$ be the neutral element and extend the operation by distributivity and linearity to all quaternions $q = x_0 + x_1 i + x_2 j + x_3 k$.  The \emph{conjugate} of such a $q$ is defined as $\bar q = x_0 - x_1 i - x_2 j - x_3 k$, its \emph{real} and \emph{imaginary part} as $Re(q) = x_0$ and $Im(q) = x_1 i + x_2 j + x_3 k$, and its \emph{modulus} as $|q|= \sqrt{q\bar q} = \sqrt{Re(q)^2 + |Im(q)|^2}$. The multiplicative inverse of each $q \neq 0$ is computed as 
$$q^{-1} = \frac{\bar q}{|q|^2}.$$

Much literature has been devoted to the possibility of defining for quaternionic functions a notion of regularity playing the same role as holomorphy for complex functions. This is by no means an elementary question, as neatly explained in \cite{sudbery}.  First of all, let us consider the following fact.

\begin{theorem}\label{derivative}
A function $f:\hh \to \hh$ is left $q$-differentiable at every $q \in \hh$, i.e.
$$\lim_{h \to 0} h^{-1}[f(q+h)-f(q)]$$ 
exists at every $q \in \hh$, if and only if there exist $a,b \in \hh$ such that $f(q) = qa+b$ for all $q \in \hh$.
\end{theorem}

A proof can be found in \cite{sudbery}. The class of functions encompassed does not grow significantly if we vary the domain of definition or consider the analogous notion of right $q$-differentiability. The next corollary can be easily derived.

\begin{corollary}
The class of functions $f:\hh \to \hh$ admitting at every $q_0 \in \hh$ a power series expansion 
$$f(q) = \sum_{n \in \nn} (q-q_0)^n a_n$$
with $\{a_n\}_{n \in \nn}$ coincides with the class of affine functions $f(q) = qa+b$ with $a,b \in \hh$.
\end{corollary}

Once again, the situation changes but little if we consider an expansion of type $f(q) = \sum_{n \in \nn} a_n (q-q_0)^n$ instead. It should be noticed, however, that $(q-q_0)^n a_n$ and $a_n (q-q_0)^n$ are not the only possible terms of degree $n$ in $q-q_0$: in fact, in an (associative) non-commutative framework, the generic form of such a term is
\begin{equation}\label{skewmonomial}
\alpha_0 (q-q_0) \alpha_1 (q-q_0) \ldots \alpha_{n-1} (q-q_0)\alpha_n.
\end{equation}
We may try to define a notion of analyticity based on series of terms of this type, but we are immediately discouraged by the following fact.

\begin{proposition}
The class of functions $f:\hh \to \hh$ admitting at every $q_0 \in \hh$ an expansion into series of monomials of type \eqref{skewmonomial} coincides with the class of functions $f:\hh \to \hh$ that are analytic in the four real variables $x_0,\ldots,x_3$.
\end{proposition}

The proof of the preceding proposition relies upon the fact that if $q=x_0+i x_1+j x_2+ k x_3$, then
\begin{eqnarray}
x_0=\frac{1}{4}(q-i qi-j qj-kqk),\ x_1=\frac{1}{4i}(q-i qi+j qj+kqk), \nonumber \\ 
x_2=\frac{1}{4j}(q+i qi-j qj+kqk),\ x_3=\frac{1}{4k}(q+i qi+j qj-kqk).\nonumber
\end{eqnarray}

The facts just mentioned encouraged to base the study of quaternionic functions on analogs of the Cauchy-Riemann equations rather than on some notion of analyticity. The best known of such analogs is due to Fueter \cite{fueter1,fueter2} and it gave rise to a renowned function theory on Clifford algebras, for which we refer the reader to \cite{librosommen,librodaniele,libroshapiro} and references therein. Let us also mention the work \cite{laville2} on the specific topic of analyticity in this context.

A different notion of quaternionic analyticity arose from the theory introduced by Gentili and Struppa in \cite{cras,advances}. They gave the following definition, where $\s = \{q \in \hh : q^2 = -1\}$ denotes the $2$-sphere of \emph{imaginary units}.

\begin{definition}
Let $\Omega$ be a domain in $\hh$ and let $f : \Omega \to \hh$ be a function. For all $I \in \s$, let us denote $L_I = \rr + I \rr$, $\Omega_I = \Omega \cap L_I$ and $f_I = f_{|_{\Omega_I}}$. 
The function $f$ is called \emph{(slice) regular} if, for all $I \in \s$, the restriction $f_I$ is holomorphic, i.e. the function $\bar \partial_I f : \Omega_I \to \hh$ defined by
$$
\bar \partial_I f (x+Iy) = \frac{1}{2} \left( \frac{\partial}{\partial x}+I\frac{\partial}{\partial y} \right) f_I (x+Iy)
$$
vanishes identically.
\end{definition}

The definition immediately implies that any power series $f(q) = \sum_{n \in \nn} q^n a_n$ defines a regular function on its ball of convergence 
$$B(0,R) = \{q \in \hh : |q| <R\}$$ 
(a perfect analog of Abel's theorem holds). Moreover, the set of such series forms a real algebra when endowed with the usual addition $+$ and the multiplication $*$ defined in the following manner:
\begin{equation}\label{product}
\left(\sum_{n \in \nn} q^n a_n \right)*\left( \sum_{n \in \nn} q^n b_n \right) = \sum_{n \in \nn} q^n \sum_{k = 0}^n a_k b_{n-k}.
\end{equation}
In \cite{powerseries}, we considered series of the form
\begin{equation}\label{regularseries}
f(q) = \sum_{n \in \nn} (q-q_0)^{*n} a_n
\end{equation}
where $(q-q_0)^{*n} = (q-q_0)*\ldots*(q-q_0)$ denotes the $*$-product of $n$ copies of $q \mapsto q-q_0$. We were able to prove that the sets of convergence of such series are balls with respect to a distance $\sigma : \hh \times \hh \to \rr$ defined in the following fashion.

\begin{definition}
For all $p,q \in \hh$, we set
\begin{equation}
\sigma(q,p) = \left\{
\begin{array}{ll}
|q-p| & \mathrm{if\ } p,q \mathrm{\ lie\ on\ the\ same\ complex\ plane\ } L_I\\
\omega(q,p) &  \mathrm{otherwise}
\end{array}
\right.
\end{equation}
where
\begin{equation}
\omega(q,p) = \sqrt{\left[Re(q)-Re(p)\right]^2 + \left[|Im(q)| + |Im(p)|\right]^2}. 
\end{equation}
\end{definition}

More precisely, we proved the next theorem.

\begin{theorem}\label{sigmaanaliticita}
If $\Omega$ is a domain in $\hh$, a function $f : \Omega \to \hh$ is regular if and only if it is \emph{$\sigma$-analytic}, i.e. it admits at every $q_0 \in \Omega$ an expansion of type \eqref{regularseries} that is valid in a $\sigma$-ball $\Sigma(q_0,R) = \{q \in \hh : \sigma(q,q_0) <R\}$.
\end{theorem}

The analogy with the complex case is remarkable, but it should be taken into account that the topology induced by $\sigma$ is finer than the Euclidean: if $q_0 = x_0+Iy_0$ does not lie on the real axis then for $R<2y_0$ the $\sigma$-ball $\Sigma(q_0,R)$ reduces to a ($2$-dimensional) disk $\{z \in L_I : |z-q_0|<R\}$ in the complex plane $L_I$ through $q_0$ (see \cite{powerseries} for a detailed account on $\sigma$-balls). Hence the series expansion \eqref{regularseries}, in general, may not predict the behavior of $f$ in a Euclidean neighborhood of $q_0$, but only along the complex plane $L_I$ containing $q_0$. 
%
This curious phenomenon is partly explained by the fact that for a generic domain $\Omega$ in $\hh$, a regular function $f : \Omega \to \hh$ needs not be continuous, as shown by the next example.

\begin{example}
For a fixed $I \in \s$, we can define a regular function $f: \hh \setminus \rr \to \hh$ by setting
$$f(q) = \left\{ 
\begin{array}{ll}
0 \ \mathrm{if} \ q \in \hh \setminus L_I\\
1 \ \mathrm{if} \ q \in L_I \setminus \rr
\end{array}
\right.$$
Clearly, $f$ is not continuous.
\end{example}

However, \cite{advancesrevised} explained that real differentiability (and other interesting properties) are granted if $\Omega$ is carefully chosen. Firstly, let us consider the following class of domains.

\begin{definition}
Let $\Omega$ be a domain in $\hh$, intersecting the real axis. If,  for all $I \in \s$, $\Omega_I = \Omega \cap L_I$ is a domain in $L_I \simeq \cc$ then $\Omega$ is called a \textnormal{slice domain}.
\end{definition}

\begin{theorem}[Identity principle]
Let $\Omega$ be a slice domain and let $f,g : \Omega \to \hh$ be slice regular. Suppose that $f$ and $g$ coincide on a subset $C$ of $\Omega_I$, for some $I \in \s$. If $C$ has an accumulation point in $\Omega_I$, then $f \equiv g$ in $\Omega$.
\end{theorem}

Secondly, let us consider slice domains having the additional property of axial symmetry with respect to the real axis, i.e. those slice domains $\Omega$ such that
$$\Omega = \bigcup_{x+Iy\in \Omega} x+y\s.$$
The word \emph{symmetric} will refer to this type of symmetry throughout the paper. A very peculiar property holds for regular functions $f$ on symmetric slice domains: they are affine when restricted to a single $2$-sphere $x+y\s$. Indeed, the Representation Formula proven in \cite{advancesrevised} can be restated as follows.
%
%
\begin{theorem}\label{representationformula}
Let $f$ be a regular function on a symmetric slice domain $\Omega$ and let $x_0+y_0\s \subset \Omega$. For all $q,q_0,q_1,q_2 \in x_0+y_0\s$ with $q_1 \neq q_2$
\begin{equation}
f(q) = (q_2-q_1)^{-1} \left[\bar q_1 f(q_1)-\bar q_2 f(q_2)\right] + q (q_2-q_1)^{-1} \left[f(q_2) - f(q_1)\right]
\end{equation}
and
\begin{equation}
f(q) = f(q_0)+ (q-q_0) (q_1-q_2)^{-1} \left[f(q_1) - f(q_2)\right]
\end{equation}
where $(q_2-q_1)^{-1} \left[f(q_2) - f(q_1)\right]$ and $(q_2-q_1)^{-1} \left[\bar q_1 f(q_1)-\bar q_2 f(q_2)\right]$ do not depend on the choice of $q_1,q_2$, but only on $x_0,y_0$.
\end{theorem}

As an immediate consequence of theorem \ref{representationformula}, if $f$ is a regular function on a symmetric slice domains then its values can all be recovered from those of one of its restrictions $f_I$. This allowed a further study of these functions in \cite{advancesrevised}, including the construction of a structure of real algebra for regular functions on a symmetric slice domain $\Omega$ (with the usual addition $+$ and a multiplication $*$ that extends the one defined in equation \eqref{product}) and the observation that a regular $f$ on a symmetric slice domain is $C^{\infty}(\Omega)$. This last result that has been improved to $f \in C^{\omega}(\Omega)$ in \cite{perotti}. This makes it reasonable to expect, in the case of symmetric slice domains, a stronger form of analyticity than the one presented in theorem \ref{sigmaanaliticita}. Finding this stronger notion of analyticity is the aim of the present paper, which introduces a new type of series expansions and proves their validity in Euclidean open sets.

We start  in section \ref{sectionseries} with a regular function $f$ on a symmetric slice domain $\Omega$, and we construct a formal expansion of the form
\begin{equation}\label{expansion0}
f(q) = \sum_{n \in \nn}[(q-x_0)^2+y_0^2]^n [A_{2n} + (q-q_0)A_{2n+1}]
\end{equation}
at each point $q_0 = x_0 + I y_0 \in\Omega$. We then prove what follows.

\begin{proposition}
Let $\{a_n\}_{n \in \nn}\subset \hh$ and suppose 
\begin{equation}
\limsup_{n \to +\infty} |a_n|^{1/n} = 1/R
\end{equation}
for some $R >0$. Let $q_0 = x_0+Iy_0 \in \hh$ with $x_0 \in \rr, y_0>0, I \in \s$ and set $P_{2n}(q) = [(q-x_0)^2+y_0^2]^n$ and $P_{2n+1}(q) = [(q-x_0)^2+y_0^2]^n(q-q_0)$ for all $n \in \nn$. Then the function series $\sum_{n \in \nn}P_n(q) a_n$ converges absolutely and uniformly on compact sets in 
\begin{equation}
U(x_0+y_0\s,R) = \{q \in \hh : |(q-x_0)^2+y_0^2| < R^2\},
\end{equation}
where it defines a regular function. Furthermore, the series diverges at every $q \in \hh \setminus \overline{U(x_0+y_0\s,R)}$.
\end{proposition}

In section \ref{sectionintegral}, we obtain estimates for the coefficients of our new formal expansion \eqref{expansion0}. More precisely, we prove that for every $U = U(x_0+y_0\s,R)$ such that $\overline{U} \subset \Omega$, there exists a constant $C>0$ such that
$$|A_n| \leq C \cdot \frac{\max_{\partial U}|f|}{R^{n}}.$$ 
This result is achieved by means of new integral representations.

In section \ref{sectionanalytic} we define $f$ to be \emph{symmetrically analytic} if it admits at any $q_0 \in \Omega$ an expansion of type \eqref{expansion0} valid in a neighborhood of $q_0$. We are able to prove that regularity is equivalent to symmetric analyticity.

In section \ref{sectionmultiplicities}, we apply the new series expansion \eqref{expansion0} to the computation of the multiplicities of the zeros of regular functions. Finally, in section \ref{sectiondifferential}, we apply it to the computation of directional derivatives.

\begin{theorem}
Let $f$ be a regular function on a symmetric slice domain $\Omega$, and let $q_0 = x_0+Iy_0 \in \Omega$. For all $v \in \hh, |v|=1$ the derivative of $f$ along $v$ can be computed at $q_0$ as
\begin{equation}
\lim_{t\to 0} \frac{f(q_0+tv)-f(q_0)}{t} = v A_1 + (q_0v - v\bar q_0)A_2.
\end{equation}
\end{theorem}

We conclude by studying the bearings of this result in (real and) complex coordinates: it turns out that if we fix $q_0$ and choose appropriate complex coordinates on $\hh$, then $f$ is complex differentiable at $q_0$ and its Jacobian at the same point can be easily computed in terms of $A_1$ and $A_2$.



\section{A new series expansion}\label{sectionseries}

Let us recall a result from \cite{singularities} (extending \cite{zeros}) and derive from it the subsequent theorem.

\begin{proposition}
Let $f$ be a regular function on a symmetric slice domain $\Omega$. A point $q_0 \in \Omega$ is a zero of $f$ if and only if there exists a regular function $g : \Omega \to \hh$ such that
$f(q) = (q-q_0)*g(q).$
\end{proposition}

\begin{theorem}
Let $f$ be a regular function on a symmetric slice domain $\Omega$. For each $q_0 \in \Omega$, let us denote as $R_{q_0}f : \Omega \to \hh$ the function such that
$$f(q) = f(q_0) + (q-q_0)*R_{q_0}f(q).$$
If $q_0 = x_0 + I y_0$ for $x_0,y_0 \in \rr$ and for some $I \in \s$, then
$$f(q)= f(q_0) + (q-q_0) R_{q_0}f(\bar q_0) +  [(q-x_0)^2+y_0^2] R_{\bar q_0}R_{q_0}f(q).$$
for all $q \in \Omega$.
\end{theorem}

\begin{proof}
The existence of a regular $R_{q_0}f : \Omega \to \hh$ such that
$$f(q) = f(q_0) + (q-q_0)*R_{q_0}f(q).$$
is granted by the previous proposition, since $f-f(q_0)$ is a regular function on $\Omega$ vanishing at $q_0$.
Applying the same procedure to $R_{q_0}f$ at the point $\bar q_0$ yields
$$f(q) = f(q_0) + (q-q_0)*\left[R_{q_0}f(\bar q_0) + (q-\bar q_0)*R_{\bar q_0}R_{q_0}f(q)\right] = $$
$$= f(q_0) + (q-q_0) R_{q_0}f(\bar q_0) +  [(q-x_0)^2+y_0^2] R_{\bar q_0}R_{q_0}f(q)$$
where we have taken into account that
$$(q-q_0)*(q-\bar q_0) = q^2-q(q_0+\bar q_0) + q_0\bar q_0 = q^2-q2x_0 + x_0^2+y_0^2 = (q-x_0)^2+y_0^2.$$
\end{proof}

If we repeatedly apply the previous theorem, we get the formal expansion
\begin{eqnarray}\nonumber
&&f(q)=f(q_0) + \left(q-q_0\right) R_{q_0}f(\bar q_0) +\\ \nonumber
&& +[(q-x_0)^2+y_0^2] \left[R_{\bar q_0}R_{q_0}f(q_0) + (q-q_0) R_{q_0}R_{\bar q_0}R_{q_0}f(\bar q_0) \right]+\ldots+\\ \nonumber
&& + [(q-x_0)^2+y_0^2]^n \left[(R_{\bar q_0}R_{q_0})^{n}f(q_0) + (q-q_0) R_{q_0}(R_{\bar q_0}R_{q_0})^{n}f(\bar q_0) \right] + \ldots
\end{eqnarray}
where $(R_{\bar q_0}R_{q_0})^{n}$ denotes the $n$th iterate of $R_{\bar q_0}R_{q_0}$.
If $A_{2n} = (R_{\bar q_0}R_{q_0})^{n}f(q_0)$ and $A_{2n+1} = R_{q_0}(R_{\bar q_0}R_{q_0})^{n}f(\bar q_0)$ for all $n \in \nn$, our new formal expansion reads as
\begin{equation}\label{expansion}
f(q) = \sum_{n \in \nn}P_n(q) A_n
\end{equation}
where $P_{2n}(q) = [(q-x_0)^2+y_0^2]^n$ and $P_{2n+1}(q) = [(q-x_0)^2+y_0^2]^n(q-q_0)$ for all $n \in \nn$. Let us study the sets of convergence of function series of this type. 

\begin{lemma}
Let $x_0 \in \rr, y_0>0, q \in \hh$ and let $r\geq0$ be such that $|(q-x_0)^2+y_0^2| = r^2$. If $q_0 = x_0 + I y_0$ for some $I \in \s$ then
$$\sqrt{r^2+y_0^2}-y_0 \leq |q-q_0| \leq \sqrt{r^2+y_0^2}+y_0$$
\end{lemma}

\begin{proof}
If $r=0$, i.e. $q \in x_0+y\s$, then clearly $0 \leq  |q-q_0| \leq 2y_0$. Else $q \not \in x_0+y\s$ and
$$|(q-x_0)^2+y_0^2| = |q-q_0| |(q-q_0)^{-1}q(q-q_0) - \bar q_0| = |q-q_0||q-\tilde{q}_0|$$
where $\tilde{q}_0 = (q-q_0)\bar q_0(q-q_0)^{-1} \in x_0 +y_0 \mathbb{S}$.
If $|q-q_0| > \sqrt{r^2+y_0^2}+y_0$ then 
$$|q-\tilde{q}_0| \geq |q-q_0|-|q_0-\tilde{q}_0| \geq |q-q_0| - 2y_0 > \sqrt{r^2+y_0^2} -y_0$$
so that
$$|(q-x_0)^2+y_0^2| = |q-q_0||q-\tilde{q}_0|> r^2+y_0^2-y_0^2 = r^2$$
a contradiction with the hypothesis. A similar reasoning excludes that $|q-q_0| < \sqrt{r^2+y_0^2}-y_0$.
\end{proof}

\begin{proposition}\label{convergence}
Let $\{a_n\}_{n \in \nn}\subset \hh$ and suppose 
\begin{equation}
\limsup_{n \to +\infty} |a_n|^{1/n} = 1/R
\end{equation}
for some $R >0$. Let $q_0 = x_0+Iy_0 \in \hh$ with $x_0 \in \rr, y_0>0, I \in \s$ and set $P_{2n}(q) = [(q-x_0)^2+y_0^2]^n$ and $P_{2n+1}(q) = [(q-x_0)^2+y_0^2]^n(q-q_0)$ for all $n \in \nn$. Then the function series 
\begin{equation}\label{series}
\sum_{n \in \nn}P_n(q) a_n
\end{equation}
converges absolutely and uniformly on compact sets in 
\begin{equation}
U(x_0+y_0\s,R) = \{q \in \hh : |(q-x_0)^2+y_0^2| < R^2\},
\end{equation}
where it defines a regular function. Furthermore, the series \eqref{series} diverges at every $q \in \hh \setminus \overline{U(x_0+y_0\s,R)}$.
\end{proposition}

\begin{proof}
If $K$ is a compact subset of $U(x_0+y_0\s,R)$ then there exists $r<R$ such that $|(q-x_0)^2+y_0^2| \leq r^2$ for all $q \in K$. Thus, for all $q \in K$
$$|P_{2n}(q) a_{2n}| = |(q-x_0)^2+y_0^2|^n |a_{2n}| \leq r^{2n} |a_{2n}|$$
while (thanks to the previous lemma)
$$|P_{2n+1}(q) a_{2n+1}| = |(q-x_0)^2+y_0^2|^n |q-q_0| |a_{2n+1}| \leq$$
$$\leq r^{2n} \left(\sqrt{r^2+y_0^2}+y_0\right) |a_{2n+1}|.$$
Hence \eqref{series} is dominated on $K$ by a number series $\sum_{n \in \nn} c_n$ with 
$$\limsup_{n \to +\infty} |c_n|^{1/n} = r/R<1.$$ 
This guarantees absolute and uniform convergence of \eqref{series} in $K$. This, in turn, proves the regularity of the sum (since all the addends in \eqref{series} are regular polynomials).

Finally, if $q \in \hh \setminus \overline{U(x_0+y_0\s,R)}$ then $|(q-x_0)^2+y_0^2| = r^2$ for some $r>R$. Reasoning as before, we get $|P_{2n}(q) a_{2n}| \geq r^{2n} |a_{2n}|$ and 
$$|P_{2n+1}(q) a_{2n+1}| \geq r^{2n} \left(\sqrt{r^2+y_0^2}-y_0\right) |a_{2n+1}|.$$ Thus, $\sum_{n \in \nn}P_n(q) a_n$ dominates a number series $\sum_{n \in \nn} C_n$ with 
$$\limsup_{n \to +\infty} |C_n|^{1/n} = r/R>1$$ 
and it must diverge.
\end{proof}

\begin{figure}[h]
\begin{center}
\includegraphics[width=.5\textwidth]{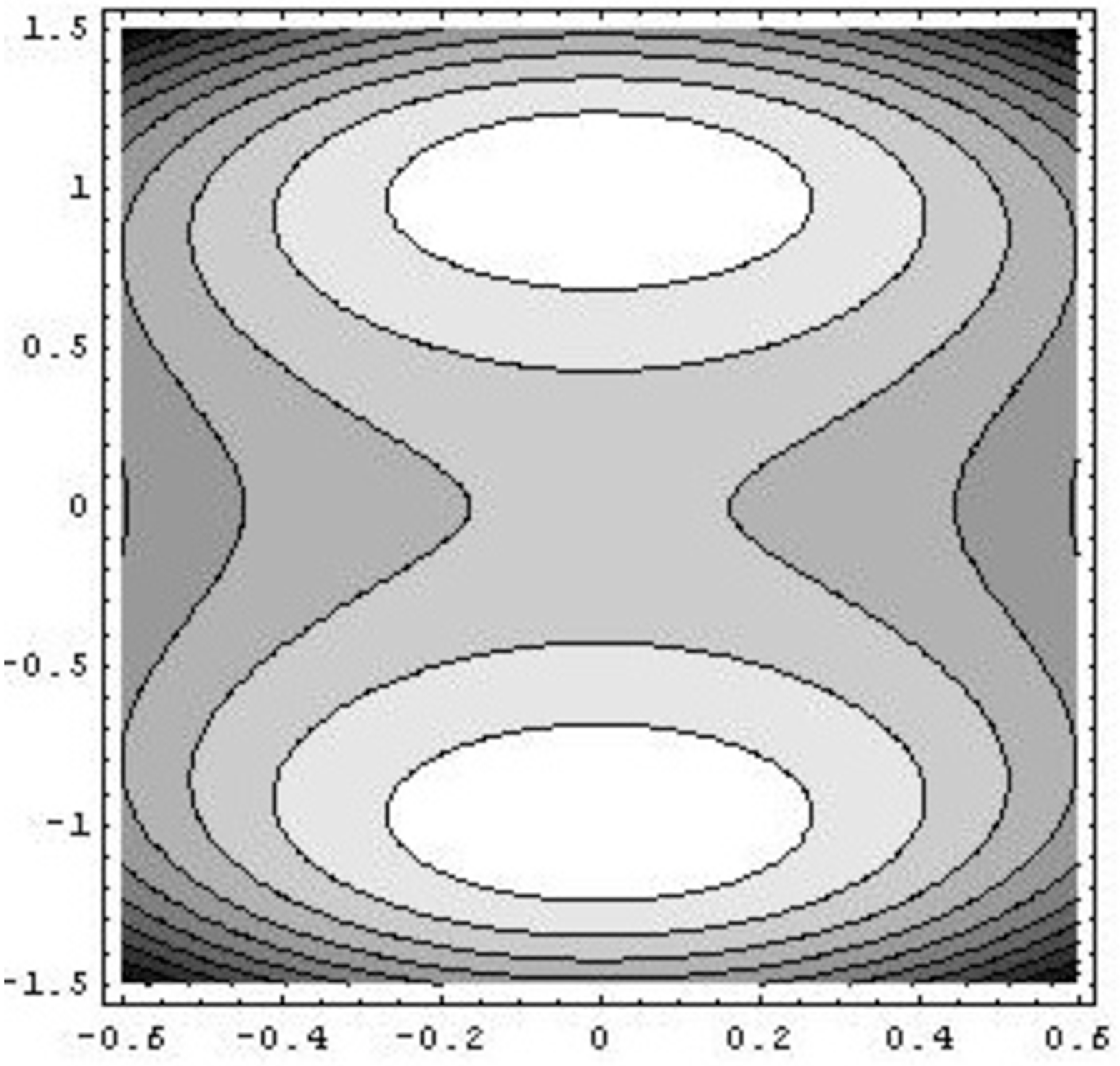}
\end{center}
\caption{Examples of $U(\s,R)$, intersected with a plane $L_I$, for different values of $R$.}
\end{figure}

Let us now describe the sets $U(x_0+y_0\s,R)$. The intersection between the hypersurface $|(q-x_0)^2+y_0^2| = R^2$ and any complex plane $L_I$ is a polynomial lemniscate, connected for $R\geq y_0$ and having two connected components for $R<y_0$. If $y_0\neq 0$ then at the critical value $R=y_0$ the lemniscate is of figure-eight type. Some examples are portrayed in figure 1. 

\begin{remark}
For $R>y_0$ the set $U(x_0+y_0\s,R)$ is a symmetric slice domain; for $0<R\leq y_0$ the set $U(x_0+y_0\s,R)$ is a symmetric domain whose intersection with any complex plane $L_I$ has two connected components.
\end{remark}

So far, we have introduced the formal expansion \eqref{expansion}, and we have studied what its set of convergence would be if estimates for the coefficients $A_n$ were provided. In the next section, we will search exactly for such estimates by means of Cauchy-type integral representations.


\section{Integral formulas and Cauchy estimates}\label{sectionintegral}

Let us recall the ``slicewise'' Cauchy integral formula proven in \cite{advances}. We first introduce some notations. Suppose ${\gamma_I} : [0,1] \to L_I$ to be a rectifiable curve whose support lies in a complex plane $L_I$ for some $I \in \s$, let $\Gamma_I$ be a neighborhood of ${\gamma_I}$ in $L_I$ and let $f,g : \Gamma_I \to \hh$ be continuous functions. If $J \in \s$ is such that $J \perp I$ then $L_I + L_IJ = \hh =L_I + J L_I$ and there exist continuous functions $F,G,H,K:\Gamma_I \to L_I$ such that $f = F+GJ$ and $g = H+JK$ in $\Gamma_I$. Then
$$\int_{\gamma_I} g(s) ds f(s) := \int_{\gamma_I} H(s) ds F(s) +  \int_{\gamma_I} H(s) dsG(s) J+$$
$$ +  J \int_{\gamma_I} K(s)ds F(s)+ J \int_{\gamma_I} K(s)ds G(s)J.$$
The aforementioned Cauchy-type formula reads as follows.

\begin{lemma}
Let $f$ be a regular function on a symmetric slice domain $\Omega$, let $I \in \s$ and let $U_I$ be a bounded Jordan domain in $L_I$, with $\overline{U_I} \subset \Omega_I$. If $\partial U_I$ is rectifiable then 
\begin{equation}
f(z) = \frac{1}{2 \pi I} \int_{\partial U_I} \frac{ds}{s-z} f(s)
\end{equation}
for all $z \in U_I$.
\end{lemma}


The previous lemma was the basis for the proof of a more complete Cauchy formula in \cite{cauchy}. It will now provide a new integral representation and lead us to the desired estimates for the coefficients $A_n$ in the formal expansion \eqref{expansion}.

\begin{theorem}
Let $f$ be a regular function on a symmetric slice domain $\Omega$, let $I \in \s$ and let $U_I$ be a symmetric bounded Jordan domain in $L_I$, with $\overline{U_I} \subset \Omega_I$. If $\partial U_I$ is rectifiable then for each $z_0 = x_0+Iy_0 \in U_I$ and for all $z \in U_I$
\begin{equation}\label{prima}
f(z) = f(z_0) + (z-z_0) \cdot \frac{1}{2 \pi I} \int_{\partial U_I} \frac{ds}{(s-z)(s-z_0)} f(s).
\end{equation}
Furthermore,
\begin{eqnarray}
(R_{\bar z_0}R_{z_0})^{n}f(z) &=& \frac{1}{2 \pi I} \int_{\partial U_I} \frac{ds}{(s-z)[(s-x_0)^2+y_0^2]^n} f(s) \label{seconda}\\
R_{z_0}(R_{\bar z_0}R_{z_0})^{n}f(z) &=& \frac{1}{2 \pi I} \int_{\partial U_I} \frac{ds}{(s-z)(s-z_0)[(s-x_0)^2+y_0^2]^{n}} f(s) \label{terza}
\end{eqnarray}
for all $n \in \nn$ and all $z \in U_I$.
\end{theorem}

\begin{proof}
The previous lemma is equivalent to \eqref{seconda} in the special case $n=0$. That immediately implies 
$$R_{z_0}f(z)=(z-z_0)^{-1}[f(z) - f(z_0)]=$$
$$= \frac{1}{2 \pi I} \int_{\partial U_I} \frac{1}{z-z_0}\left[ \frac{1}{s-z} - \frac{1}{s-z_0}\right]ds f(s)=$$
$$=\frac{1}{2 \pi I} \int_{\partial U_I} \frac{ds}{(s-z)(s-z_0)} f(s),$$
that is \eqref{terza} in the special case $n=0$. This last equality, in turn, implies \eqref{prima}.
Moreover,
$$R_{z_0}f(\bar z_0)= \frac{1}{2 \pi I} \int_{\partial U_I} \frac{ds}{(s-x_0)^2+y_0^2} f(s)$$
and
$$R_{\bar z_0}R_{z_0}f(z) = (z-\bar z_0)^{-1}[R_{z_0}f(z) - R_{z_0}f(\bar z_0)]=$$
$$= \frac{1}{2 \pi I} \int_{\partial U_I} \frac{1}{z-\bar z_0}\left[ \frac{1}{(s-z)(s-z_0)} - \frac{1}{(s-x_0)^2+y_0^2}\right]ds f(s)=$$
$$= \frac{1}{2 \pi I} \int_{\partial U_I} \frac{1}{z-\bar z_0}\frac{s-\bar z_0 - (s-z)}{(s-z)[(s-x_0)^2+y_0^2]}ds f(s)=$$
$$=\frac{1}{2 \pi I} \int_{\partial U_I} \frac{ds}{(s-z)[(s-x_0)^2+y_0^2]} f(s),$$
that is \eqref{seconda} with $n=1$.
An analogous reasoning proves that if \eqref{seconda} holds for $n=k$, then \eqref{terza} holds for $n=k$, which in turn implies \eqref{seconda} for $n=k+1$ completing the proof.
\end{proof}

We are now in a position to get the desired estimates for the coefficients $A_n$ in the formal expansion \eqref{expansion}.

\begin{corollary}\label{estimates}
Let $f$ be a regular function on a symmetric slice domain $\Omega$, let $x_0+y_0\s \subset \Omega$ and let $q_0 \in x_0+y_0\s$. Set $A_{2n} = (R_{\bar q_0}R_{q_0})^{n}f(q_0)$ and $A_{2n+1} = R_{q_0}(R_{\bar q_0}R_{q_0})^{n}f(\bar q_0)$. For every $U = U(x_0+y_0\s,R)$ such that $\overline{U} \subset \Omega$, there exists a constant $C>0$ such that
\begin{equation}
|A_n| \leq C \cdot \frac{\max_{\partial U}|f|}{R^{n}}
\end{equation}
for all $n \in \nn$.
\end{corollary}

\begin{proof}
If we choose $U = U(x_0+y_0\s,R)$ then the previous theorem implies
$$|(R_{\bar q_0}R_{q_0})^{n}f(q_0)| = \left| \frac{1}{2 \pi I} \int_{\partial U_I} \frac{ds}{(s-z_0)[(s-x_0)^2+y_0^2]^n} f(s)\right| \leq$$
$$\leq \frac{1}{2 \pi} \int_{\partial U_I} \frac{|f(s)|}{\left(\sqrt{R^2+y_0^2}-y_0\right)R^{2n}} d|s| \leq  C \cdot \frac{\max_{\partial U}|f|}{R^{2n}}$$
if we set $C=\frac{length(\partial U_I)}{2 \pi \left(\sqrt{R^2+y_0^2}-y_0\right)}$.
On the other hand
$$|R_{q_0}(R_{\bar q_0}R_{q_0})^{n}f(\bar q_0)| = \left| \frac{1}{2 \pi I} \int_{\partial U_I} \frac{ds}{[(s-x_0)^2+y_0^2]^{n+1}} f(s)\right| \leq$$
$$\leq \frac{1}{2 \pi} \int_{\partial U_I} \frac{|f(s)|}{R^{2n+2}} d|s| \leq K \cdot \frac{\max_{\partial U}|f|}{R^{2n+1}}$$
where $K= \frac{length(\partial U_I)}{2\pi R}\leq C$.
\end{proof}


\section{A new notion of analyticity}\label{sectionanalytic}

We are now ready to prove the desired result concerning the expansion of slice regular functions.

\begin{theorem}
Let $f$ be a regular function on a symmetric slice domain $\Omega$, and let $x_0,y_0 \in \rr$ and $R >0$ be such that $U(x_0+y_0\s,R) \subseteq \Omega$. For all $q_0 \in x_0+y_0\s$, setting $A_{2n} = (R_{\bar q_0}R_{q_0})^{n}f(q_0)$ and $A_{2n+1} = R_{q_0}(R_{\bar q_0}R_{q_0})^{n}f(\bar q_0)$, we have that
\begin{equation}\label{expansion2}
f(q) = \sum_{n \in \nn}[(q-x_0)^2+y_0^2]^n [A_{2n} + (q-q_0)A_{2n+1}]
\end{equation}
for all $q \in U(x_0+y_0\s,R)$.
\end{theorem}

\begin{proof}
Thanks to proposition \ref{convergence} and to corollary \ref{estimates}, the function series in equation \eqref{expansion2} converges in $U=U(x_0+y_0\s,R)$, where it defines a regular function. Let us consider the difference 
$$g(q) = f(q) -\sum_{n \in \nn}[(q-x_0)^2+y_0^2]^n [A_{2n} + (q-q_0)A_{2n+1}].$$ 
By construction, $g(q) = [(q-x_0)^2+y_0^2]^n(R_{\bar q_0}R_{q_0})^{n}f(q)$ for all $n \in \nn$.
For any choice of $I \in \s$ and for all $z \in U_I$, we derive that 
$$g_I(z) = [(z-x_0)^2+y_0^2]^n(R_{\bar q_0}R_{q_0})^{n}f_I(z) = [z-(x_0+Iy_0)]^n h^{[n]}_I(z)$$ 
where $h^{[n]}_I(z) = [z-(x_0-Iy_0)]^n (R_{\bar q_0}R_{q_0})^{n}f_I(z)$ is holomorphic in $U_I$. The identity principle for holomorphic functions of one complex variable implies that $g_I \equiv0$. Since $I$ can be arbitrarily chosen in $\s$, the function $g$ must be identically zero in $U$. This is equivalent to equation \eqref{expansion2}.
\end{proof}

\begin{definition}
Let $f$ be a regular function on a symmetric slice domain $\Omega$. We say that  $f$ is \emph{symmetrically analytic} if it admits at any $q_0 \in \Omega$ an expansion of type \eqref{expansion} valid in a neighborhood of $q_0$.
\end{definition}

The previous theorem, along with proposition \ref{convergence}, proves what follows.

\begin{corollary}
Let $\Omega$ be a symmetric slice domain. A function $f: \Omega \to \hh$ is regular if, and only if, it is symmetrically analytic.
\end{corollary}

We conclude this section reformulating expansion \eqref{expansion2} as follows, thanks to theorem \ref{representationformula}.

\begin{corollary}
Let $f$ be a regular function on a symmetric slice domain $\Omega$, and let $x_0,y_0 \in \rr$ and $R >0$ be such that $U(x_0+y_0\s,R) \subseteq \Omega$. Then, for all $q \in U(x_0+y_0\s,R)$
\begin{equation}\label{expansion3}
f(q) = \sum_{n \in \nn}[(q-x_0)^2+y_0^2]^n [C_{2n} + qC_{2n+1}]
\end{equation}
where the coefficients $C_{n}$ depend only on $f,x_0,y_0$ and can be computed as
\begin{eqnarray}
&&C_{2n} =(q_2-q_1)^{-1} \left[\bar q_1 (R_{\bar q_0}R_{q_0})^{n}f(q_1)-\bar q_2 (R_{\bar q_0}R_{q_0})^{n}f(q_2)\right],\\
&&C_{2n+1} = (q_2-q_1)^{-1} \left[(R_{\bar q_0}R_{q_0})^{n}f(q_2) - (R_{\bar q_0}R_{q_0})^{n}f(q_1)\right]\nonumber
\end{eqnarray}
for all $q_0,q_1,q_2 \in x_0+y_0\s$ with $q_1 \neq q_2$.
\end{corollary}

We notice that the odd-indexed coefficients in expansions \eqref{expansion2} and \eqref{expansion3} coincide.


\section{Computing multiplicities of zeros}\label{sectionmultiplicities}

As a first application of our new type of series expansions, let us look at its bearings in the computation of the multiplicities of the zeros. After the study of the roots of polynomials of \cite{shapiro}, the following result was proven for power series in \cite{zeros}, and in the following form in \cite{zerosopen} (for generalizations to other real alternative algebras, see \cite{perotti,ghiloni}).

\begin{theorem}[Structure of the Zero Set]
Let $f$ be a regular function on a symmetric slice domain $\Omega$. If $f$ does not vanish identically, then the zero set of $f$ consists of isolated points or isolated 2-spheres of the form $x+y\mathbb{S}$.
\end{theorem}

The following notion of multiplicity was introduced in \cite{zeros}.

\begin{definition}
Let $f$ be a regular function on a symmetric slice domain $\Omega$ and let $q_0 \in \Omega$. We define the \emph{(classical) multiplicity} of $q_0$ as a zero of $f$ and denote by $m_f(q_0)$ the largest $n \in \nn$ such that
$f(q)=(q-q_0)^{*n}*g(q)$ for some regular $g: \Omega \to \hh$.
\end{definition}

In other words, $m_f(q_0)$ is the index $n$ of the first non vanishing  coefficient $a_n$ in the expansion \eqref{regularseries}. Despite its analogy with complex multiplicity, this notion of multiplicity is not completely satisfactory, since the sum of the classical multiplicities of the zeros of a polynomial is unrelated to its degree.
\begin{example}
Let us choose $I\in \s$ and let
$$P(q) = (q-I)*(q+I) = q^2+1.$$ 
$P$ has multiplicity $m_P(I) = 1$ at all $I \in \s$.
\end{example}

\begin{example}
Let us choose $I,J \in \s$ with $I \neq -J$ and let
$$P(q) = (q-I)*(q-J) = q^2-q(I+J)+IJ.$$
The zero set of $P$ is $\{I\}$. If $I \neq J$ then $m_P(I)=1$, while $f$ has degree $2$.
\end{example}

For this reason \cite{milan} introduced alternative notions of multiplicity for the roots of regular polynomials. We present them as generalized  in \cite{singularities} to all regular functions on symmetric slice domains.

\begin{theorem}\label{factorization}
Let $f$ be a regular function on a symmetric slice domain $\Omega$, suppose $f \not \equiv 0$ and let $x+y\s \subset\Omega$. There exist $m \in \nn$ and a regular function $\tilde f:\Omega \to \hh$ not identically zero in $x+y\s$ such that
\begin{equation}
f(q) = [(q-x)^2+y^2]^m \tilde f(q).
\end{equation}
If $\tilde f$ has a zero $p_1 \in x+y\s$ then such a zero is unique and there exist $n \in \nn$, $p_2,...,p_n \in x+y\s$ (with $p_i \neq \bar p_{i+1}$ for all $i \in \{1,\ldots,n-1\}$) such that
\begin{equation}
\tilde f(q) = (q-p_1)*(q-p_2)*...*(q-p_n)*g(q)
\end{equation}
for some regular function $g:\Omega \to \hh$ which does not have zeros in $x+y\s$.
\end{theorem}

\begin{definition}
In the situation of Theorem \ref{factorization}, we say that $f$ has \emph{spherical multiplicity} $2m$ at $x+y\s$ and that $f$ has \emph{isolated multiplicity} $n$ at $p_1$.
\end{definition}

As observed in \cite{milan}, the degree of a polynomial equals the sum of the spherical multiplicities and of the isolated multiplicities of its zeros. For instance: in the previous example, $q^2+1$ has spherical multiplicity $2$ at $\s$; $(q-I)*(q-J)$ (with $I \neq -J$) has spherical multiplicity $0$ at $\s$ and isolated multiplicity $2$ at $I$.

Our new series expansion allows an immediate computation of the spherical multiplicity and it gives some information on the isolated multiplicity.

\begin{remark}
Let $f$ be a regular function on a symmetric slice domain $\Omega$ and let $q_0=x_0+Iy_0 \in\Omega$. If $A_{2n}$ or $A_{2n+1}$ is the first non vanishing coefficient in the expansion \eqref{expansion2} then $2n$ is the spherical multiplicity of $f$ at $x_0+y_0\s$. Moreover, $q_0$ has positive isolated multiplicity if and only if $A_{2n} =0$. 
\end{remark}

\begin{remark}
Let $f$ be a regular function on a symmetric slice domain $\Omega$ and let $x_0+y_0\s \subset \Omega$. If the first non vanishing coefficient in the expansion \eqref{expansion3} is $C_{2n}$ or $C_{2n+1}$ is  then $2n$ is the spherical multiplicity of $f$ at $x_0+y_0\s$. Moreover, there exists a point with positive isolated multiplicity in $x_0+y_0\s$ if, and only if, $C_{2n+1}^{-1}C_{2n} \in x_0+y_0\s$. 
\end{remark}


\section{Differentiating regular functions}\label{sectiondifferential}

In section \ref{sectionanalytic} we proved that a regular function $f$ on a symmetric slice domain $\Omega$ admits an expansion of the form
$$f(q) = \sum_{n \in \nn}[(q-x_0)^2+y_0^2]^n [A_{2n} + (q-q_0)A_{2n+1}]$$
at each $q_0\in \Omega$. This new expansion allows the computation of the real partial derivatives of $f$.

\begin{theorem}
Let $f$ be a regular function on a symmetric slice domain $\Omega$, and let $q_0 = x_0+Iy_0 \in \Omega$. For all $v \in \hh, |v|=1$ the derivative of $f$ along $v$ can be computed at $q_0$ as
\begin{equation}\label{differential}
\lim_{t\to 0} \frac{f(q_0+tv)-f(q_0)}{t} = v A_1 + (q_0v - v\bar q_0)A_2
\end{equation}
where $A_1=R_{q_0}f(\bar q_0),A_2 = R_{\bar q_0}R_{q_0}f(q_0)$. In particular if $e_0,e_1,e_2,e_3 \in \hh$ form a basis for $\hh$ and if $x_0,x_1,x_2,x_3$ denote the corresponding coordinates, then
\begin{equation}
\frac{\partial f}{\partial x_i}(q_0) = e_i R_{q_0}f(\bar q_0) + (q_0e_i - e_i\bar q_0)R_{\bar q_0}R_{q_0}f(q_0)
\end{equation}
\end{theorem}

\begin{proof}
We observe that
$$(q-x_0)^2+y_0^2 = (q-q_0)*(q-\bar q_0)=q(q-q_0) -(q-q_0)\bar q_0$$
and for $q = q_0+tv$
$$(q-x_0)^2+y_0^2= (q_0+tv)tv-tv\bar q_0 = t(tv^2+q_0v-v\bar q_0)$$
whence
$$f(q_0+tv) =  \sum_{n \in \nn}t^n(tv^2+q_0v-v\bar q_0)^n [A_{2n} + tvA_{2n+1}].$$
Hence,
$${f(q_0+tv)-f(q_0)}= tA_1 + t(tv^2+q_0v-v\bar q_0)[A_2 + tv A_3] +o(t)$$
and the thesis immediately follows.
\end{proof}

We notice that similar reasonings allow the computation of higher order derivatives. Furthermore, we make the following observation.

\begin{remark}
Let $f$ be a regular function on a symmetric slice domain $\Omega$ and let $q_0 \in \Omega$. If $v$ lies in the same $L_I$ as $q_0$, then $v$ commutes with $q_0$ and
$$v R_{q_0}f(\bar q_0) + (q_0v - v\bar q_0)R_{\bar q_0}R_{q_0}f(q_0) = v [R_{q_0}f(\bar q_0) + (q_0-\bar q_0)R_{\bar q_0}R_{q_0}f(q_0)] =$$
$$=  v [R_{q_0}f(\bar q_0) + R_{q_0}f(q_0)- R_{q_0}f(\bar q_0)] = vR_{q_0}f(q_0).$$
On the other hand, if $q_0 = x_0+Iy_0$ with $y_0\neq 0$ and $v$ is tangent to the $2$-sphere $x_0+y_0\s$ at $q_0$, then $q_0 v=  v\bar q_0$ and
$$v R_{q_0}f(\bar q_0) + (q_0v - v\bar q_0)R_{\bar q_0}R_{q_0}f(q_0) = vR_{q_0}f(\bar q_0).$$
\end{remark}

As a consequence, the notions of derivatives of regular functions already introduced in literature can be recovered as special cases of formula \eqref{differential}. Indeed, $R_{q_0}f(q_0)$ coincides with $\partial_cf(q_0)$, where $\partial_c f$ denotes the \emph{Cullen (or slice) derivative} of $f$, defined in \cite{advances} to equal
$$\partial_I f (x+Iy) = \frac{1}{2} \left( \frac{\partial}{\partial x}-I\frac{\partial}{\partial y} \right) f_I (x+Iy)$$
at each point $x+Iy \in \Omega_I$. On the other hand, $R_{q_0}f(\bar q_0)$ equals  $\frac{\partial_s f(q_0)}{y_0}$, where 
$\partial_s f$ denotes the \emph{spherical derivative} of $f$, defined in \cite{perotti} setting
$$\partial_s f (q_0) = \frac{1}{2}  Im(q_0)^{-1} (f(q_0)-f(\bar q_0)).$$

We conclude by looking at the implications of our result in complex coordinates. In the hypotheses of the previous theorem, let us choose $I \in \s$ so that $q_0 \in L_I$, choose $J \in \s$ with $I \perp J$ and set $e_0=1, e_1 = I, e_2 = J, e_3 = IJ$. Then $\hh = (\rr+I\rr)+(\rr+I\rr)J$ can be identified with $\cc^2$ setting $z_1 = x_0+Ix_1, z_2 = x_2+Ix_3, \bar z_1 = x_0-Ix_1, \bar z_2 = x_2-Ix_3$; we may as well split $f = f_1 + f_2 J$ for some $f_1,f_2: \Omega \to L_I$.

\begin{theorem}
Let $\Omega$ be a symmetric slice domain, let $f : \Omega \to \hh$ be a regular function and let $q_0 \in \Omega$. Chosen $I,J \in \s$ so that $q_0 \in L_I$ and $I \perp J$, let $z_1,z_2,\bar z_1, \bar z_2$ be the induced coordinates and let $\partial_1,\partial_2,\bar \partial_1, \bar \partial_2$ be the corresponding derivations. Then
\begin{equation}\label{complexholomorphy}
\left.\left( \begin{array}{cc}
\bar\partial_1 f_1 & \bar \partial_2 f_1\\
\bar \partial_1 f_2 & \bar \partial_2 f_2
\end{array}\right)\right|_{q_0} =
\left( \begin{array}{cc}
0 & 0 \\
0 & 0
\end{array}\right).
\end{equation}
Furthermore, if $R_{q_0}f$ splits as $R_{q_0}f = R_1+R_2J$ with $R_1,R_2$ ranging in $L_I$ then
\begin{equation}\label{complexjacobian}
\left.\left( \begin{array}{cc}
\partial_1 f_1 & \partial_2 f_1 \\
\partial_1 f_2 & \partial_2 f_2
\end{array}\right)\right|_{q_0} =
\left( \begin{array}{cr}
R_1(q_0) & - \overline{R_2(\bar q_0)} \\
R_2(q_0) & \overline{R_1(\bar q_0)}
\end{array}\right)
\end{equation}
\end{theorem}

\begin{proof}
Setting $e_0=1, e_1 = I, e_2 = J, e_3 = IJ$ and denoting $x_0,x_1,x_2,x_3$ the corresponding real coordinates, the derivations $\partial_1,\partial_2,\bar \partial_1, \bar \partial_2$ are defined by
$$\partial_1 =\frac{1}{2}\left(\frac{\partial}{\partial x_0} - I \frac{\partial}{\partial x_1}\right),\ \bar \partial_1 =\frac{1}{2}\left(\frac{\partial}{\partial x_0} + I \frac{\partial}{\partial x_1}\right)$$
$$\partial_2 =\frac{1}{2}\left(\frac{\partial}{\partial x_2} - I \frac{\partial}{\partial x_3}\right), \bar \partial_2 =\frac{1}{2}\left(\frac{\partial}{\partial x_2} + I \frac{\partial}{\partial x_3}\right)$$
By means of the previous proposition, we compute 
$$\frac{\partial f}{\partial x_0}(q_0) = R_{q_0}f(q_0) = R_1(q_0)+R_2(q_0)J,$$
$$\frac{\partial f}{\partial x_1}(q_0) = IR_{q_0}f(q_0) = IR_1(q_0)+IR_2(q_0)J,$$
$$\frac{\partial f}{\partial x_2}(q_0) = JR_{q_0}f(\bar q_0) + (q_0J - J\bar q_0)R_{\bar q_0}R_{q_0}f(q_0) = JR_{q_0}f(\bar q_0) =$$
$$= J (R_1(\bar q_0)+R_2(\bar q_0)J) = - \overline{R_2(\bar q_0)} + \overline{R_1(\bar q_0)}J,$$
$$\frac{\partial f}{\partial x_3}(q_0) = KR_{q_0}f(\bar q_0) + (q_0K - K\bar q_0)R_{\bar q_0}R_{q_0}f(q_0) = K R_{q_0}f(\bar q_0) =$$
$$= I J (R_1(\bar q_0)+R_2(\bar q_0)J) = - I\overline{R_2(\bar q_0)} + I\overline{R_1(\bar q_0)}J,$$
where we have taken into account that $Jz = \bar z J$ for all $z \in L_I$. The thesis follows by direct computation.
\end{proof}

Equation \eqref{complexholomorphy} tells us that fixing $q_0$ and choosing appropriate complex coordinates on $\hh$, $f$ is complex differentiable at $q_0$; equation \eqref{complexjacobian} then allows to compute the complex Jacobian of $f$ at $q_0$. These tools open the possibility of using complex variables in the study of slice regular functions.
We also point out that in the special case where $q_0 \in \rr$ the complex Jacobian has the form 
$$\left(\begin{array}{cr}
a & -\bar b \\
b & \bar a
\end{array}\right),$$ 
which implies $\lim_{h \to 0} h^{-1}[f(q_0+h)-f(q_0)] = a+bJ$ and proves what follows.

\begin{corollary}
Let $\Omega$ be a symmetric slice domain. If $f : \Omega \to \hh$ is a regular function then it is left $q$-differentiable at each $q_0 \in \Omega \cap \rr$, where  $\lim_{h \to 0} h^{-1}[f(q_0+h)-f(q_0)] =R_{q_0}f(q_0) = \partial_c f(q_0)$.
\end{corollary}


\bibliography{Expansion}

\bibliographystyle{abbrv}


\end{document}